\documentclass[12pt]{amsart}
\usepackage{amscd}
\usepackage{amssymb,amsmath}
\usepackage{hyperref}
\usepackage{mathrsfs}
\usepackage{graphicx}
\usepackage{enumitem}
\usepackage{romannum}
\usepackage{mathtools}
\usepackage[utf8]{inputenc}
\usepackage{listings}

\usepackage{lipsum,eso-pic,xcolor}
\usepackage{lineno}

\usepackage{tikz}
\usetikzlibrary{arrows,decorations.pathmorphing,backgrounds,positioning,fit}
\usetikzlibrary{positioning}
\hypersetup{
	colorlinks=true,
	linkcolor=blue,}

\usetikzlibrary{trees}
\usetikzlibrary{arrows}

\def\supp{\operatorname{supp}}

\def\reg{\operatorname{reg}}

\def\max{\operatorname{max}}

\DeclarePairedDelimiter\ev{\langle}{\rangle}

\newcommand{\m}{\mathfrak m}

\newcommand{\NN}{\mathbb{N}}

\newcommand{\PP}{\mathcal{P}}

\newtheorem{lemma}{Lemma}[section]
\newtheorem{corollary}[lemma]{Corollary}
\newtheorem{theorem}[lemma]{Theorem}
\newtheorem{proposition}[lemma]{Proposition}
\newtheorem{definition}[lemma]{Definition}
\newtheorem{remark}[lemma]{Remark}
\newtheorem{example}[lemma]{Example}

\newtheorem{conjecture}[lemma]{Conjecture}

\advance\headheight1.15pt

\begin{document}
	
	\pagenumbering{arabic}
	
	\title[Componentwise Linear Weighted Oriented Edge Ideals]{Componentwise linearity of edge ideals of weighted oriented Graphs} 
	
	\author[Manohar Kumar]{Manohar Kumar$^a$}
	\address{Department of Mathematics, Indian Institute of Technology
		Kharagpur, West Bengal, INDIA - 721302.}
	\email{kumarmanohar.maths@gmail.com}
	
	\author[Ramakrishna Nanduri]{Ramakrishna Nanduri$^{b}$}
	\address{Department of Mathematics, Indian Institute of Technology
		Kharagpur, West Bengal, INDIA - 721302.}
  \email{nanduri@maths.iitkgp.ac.in} 
	
	\author[Kamalesh Saha]{Kamalesh Saha$^c$}
	\address{Department of Mathematics, Chennai Mathematical Institute, Chennai, INDIA - 603103.}
	\email{ksaha@cmi.ac.in}
	
	\thanks{$^a$ Supported by PMRF fellowship, India}
	\thanks{$^b$ Supported by SERB grant No: CRG/2021/000465, India}
	\thanks{$^c$ Supported by NBHM Postdoctoral Fellowship, India}
	\thanks{AMS Classification 2020: 13D02, 05E40, 05E99, 13D45}

	\maketitle
	
	\begin{abstract}
		In this paper, we study the componentwise linearity of edge ideals of weighted oriented graphs. We show that if $D$ is a weighted oriented graph whose edge ideal $I(D)$ is componentwise linear, then the underlying simple graph $G$ of $D$ is co-chordal. This is an analogue of Fr\"oberg's theorem for weighted oriented graphs. We give combinatorial characterizations of componentwise linearity of $I(D)$ if $V^+$ are sinks or $\vert V^+ \vert\leq 1$. Furthermore, if $G$ is chordal or bipartite or $V^+$ are sinks or $\vert V^+ \vert\leq 1$, then we show the following equivalence for $I(D)$:  
  \begin{center}
  Vertex splittable $\Longleftrightarrow$ Linear quotient $\Longleftrightarrow$  Componentwise linear.  
\end{center}
	\end{abstract}
\section{Introduction}

The study of ideals having linear resolution has paid lots of attention among researchers in commutative algebra. These ideals behave nicely and computationally simple than others. One of
the remarkable results due to Eagon-Reiner says that the Stanley-Reisner ideal of a simplicial complex is Cohen-Macaulay if and only if its Alexander dual has a linear resolution \cite{ev98}. As a generalization of ideals having linear resolution, in 1999, Herzog and Hibi \cite{hh99} defined the ``next best" class of ideals, known as the componentwise linear ideals. Let $I\subseteq R=\mathbb{K}[x_1,\ldots,x_n]= \bigoplus_{d\geq 0} R_{d}$ be a graded ideal, where $\mathbb{K}$ is a field. We denote the ideal generated by all degree $d$ elements in $I$ by $I_{\ev{d}}$, i.e., $I_{\ev{d}}:=(I_{d})$. Then the ideal $I$ is said to be componentwise linear if $I_{\ev{d}}$ has a linear resolution for each $d\geq 0$.\par 

Componentwise linear ideals, a more general class including ideals having linear resolutions, have become popular in commutative algebra due to their nice properties and algebraic interpretation. Herzog and Hibi \cite{hh99} extend the Eagon-Reiner theorem by proving that a Stanley-Reisner ideal $I$ is sequentially Cohen-Macaulay if and only if its Alexander dual $I^{\vee}$ is componentwise linear. One of the motivations to study componentwise linear ideals comes from the literature of Koszul algebras and Koszul modules. Specifically, the ideal $I$ is Koszul if and only if $I$ is componentwise linear ideals (see \cite{hi05}). Some references on the theory of componentwise linearity are \cite{hh99}, \cite{hhm22}, \cite{hv22}, \cite{nms21}, \cite{nr15}.\par

In this paper, we study the componentwise linearity of edge ideals of weighted oriented graphs. A {\em weighted oriented graph} is a triplet $D=(V(D), E(D), w)$, where $V(D)$ is the vertex set of $D$, $E(D)=\{(x,y) |\mbox{ there is a directed edge from vertex $x$ towards vertex $y$}\}$ is the {\it edge set} of $D$, and $w:V(D)\rightarrow \NN$ is a map, called weight function, i.e. a weight $w(x)$ is assigned to each vertex $x\in V(D)$. Note that $D$ has no self-edges and no loop edges. Corresponding to a weighted oriented graph $D$, there is a simple graph $G$, called the underlying graph of $D$, such that $V(G):=V(D)$ and $E(G):=\{\{x,y\}| (x,y) \mbox { or } (y,x)\in E(D)\}$, i.e. $G$ is the simple graph without orientation and weights in $D$. Let $D$ be a weighted oriented graph with $V(D)=\{x_1,\ldots,x_n\}$. Then the {\it edge ideal} of $D$, denoted by $I(D)$, is an ideal of $R$ defined as follows
\begin{equation*}
    I(D)=(x_ix_j^{w(x_j)}~|~(x_i,x_j)\in E(D)).
\end{equation*}

The motivation for studying edge ideals of oriented graphs comes from coding theory, in particular, in the study of Reed-Muller-type codes (see \cite{hlmrv19,prt19} ). These ideals appear as the initial ideals of certain vanishing ideals in the theory of Reed-Muller-type codes. The study of weighted oriented edge ideals help to compute and estimate some basic parameters of such codes. Since the study of edge ideals of simple graphs itself is a landmark in commutative algebra and combinatorics, people try to generalize those theories for edge ideals of weighted directed graphs, which is a much bigger class. So far, a significant amount of research has been done concerning algebraic invariants and properties of these ideals (see \cite{gmsv18}, \cite{hlmrv19}, \cite{prt19}, \cite{mp21}, \cite{x21}, \cite{kblo22}, \cite{sg22}, \cite{bds23}, \cite{kn23}, \cite{kn123}). \par 

One of the most celebrated results in the study of edge ideals of simple graphs is Fr\"{o}berg's theorem \cite{f90}, which states that the edge ideal of a simple graph $G$, denoted by $I(G)$, has a linear resolution if and only if the complement of $G$ is chordal. If $D$ is a weighted oriented graph with underlying simple graph $G$, then $I(D)=I(G)$ if and only if $w(x)=1$ for all $x\in V(D)$. Recently, in \cite{bds23}, the authors showed that if $I(D)\neq I(G)$ and all weights of $D$ are equal, then $I(D)$ has linear resolution if and only if $D$ is a weighted oriented star graph such that all the edges are oriented towards the root of $D$. Thus, for non-trivial weighted oriented graphs, the class of ideals having linear resolutions is very limited. One can not simply generalize or extend characterizations of componentwise linear simple graphs to weighted oriented graphs because edge ideals of weighted oriented graphs are not square-free, not equigenerated and depend on the orientation and weights. 

Since the weighted oriented edge ideals need not be equigenerated, a question naturally arises concerning the componentwise linearity of these ideals. Also, one can ask what may be an analogue of Fr\"{o}berg's theorem for weighted oriented edge ideals. In \cite{sg22}, it has been proved that if $I(D)^{\vee}$ is Cohen-Macaulay, then $I(G)$ has a linear resolution. In this paper, we try to address these questions. The paper is written in the following fashion.\par 

In Section \ref{secpreli}, we recall some definitions, notions and results associated with our work. In section \ref{seccochordal}, we extend Fr\"{o}berg's theorem for edge ideals of weighted oriented graphs in one direction (necessary condition). In particular, we have shown in Theorem \ref{thmcochordal} that if $I(D)$ is componentwise linear, then the underlying simple graph $G$ is co-chordal (i.e., the complement of $G$ is chordal). In general, if a monomial ideal $I$ is componentwise linear, then the radical of $I$ need not be componentwise linear (see Example \ref{radical}). But, from our result, it follows that if $I(D)$ is componentwise linear, then its radical $I(G)$ has a linear resolution. Next, to find the sufficient condition of componentwise linearity for weighted oriented edge ideals, we investigate the linear quotient and vertex splitting properties of these ideals in section \ref{secvertexsplit}. The concept of linear quotient ideals was introduced by Herzog and Takayama \cite{ht02} to study resolutions by mapping cones. On the other hand, as a dual of vertex decomposable simplicial complexes, Moradi and Khosh-Ahang \cite{mka16} defined the vertex splittable ideals. From \cite[Theorem 8.2.15]{hh11}  and \cite[Theorem 2.4]{mka16}, we have
\begin{center}
    Vertex splittable $\implies$ Linear quotient $\implies$ Componentwise linear. 
\end{center}
The reverse implications need not be true because componentwise linearity depends on the characteristic of the field (for example, see \cite[Remark 3]{reisner76}), but vertex splittable and linear quotient properties are field-independent. The equivalency of the above implication is known for very few classes of ideals, namely, polymatroidal ideals \cite{ bh13}, edge ideals of simple graphs \cite{mka16}, etc. In Section \ref{secvertexsplit}, we have shown some large classes of weighted oriented edge ideals satisfy the equivalency of the relation hold. Specifically,  when $V^+$ are sinks (Theorem \ref{thmsink}) or $|V^+|=1$ (Theorem \ref{thmcardinality}) or when the underlying graph $G$ is bipartite (Theorem \ref{thmbipartite}) or chordal (Theorem \ref{thmchordal})  (where $V^{+}=\{x\in V(D)\mid w(x)>1\}$), the following equivalent relation holds: 
\begin{center}
  Vertex splittable $\Longleftrightarrow$ Linear quotient $\Longleftrightarrow$ Componentwise linear.  
\end{center}
Also, we give a combinatorial characterization of componentwise linear weighted oriented edge ideals when $V^+$ are sinks or $\vert V^{+}\vert \leq1 $. From our results and some computational evidence, we end up with Conjecture \ref{conjcl}.

\section{Preliminaries}\label{secpreli}
In this section, we recall the necessary prerequisites which are used to describe our work and establish our results.\par 

Let $D=(V(D), E(D), w)$ be a weighted oriented graph with $V(D)=\{x_1,\ldots,x_n\}$ and underlying simple graph $G$. Let $\mathfrak{m}=(x_1,\ldots,x_n)$ be the homogeneous maximal ideal of $R=\mathbb{K}[x_1,\ldots,x_n]$, where $\mathbb{K}$ is a field. We write $V^+(D):= \{x\in V(D): w(x)>1\}$, in short $V^{+}(D)$ is denoted by $V^+$. For a vertex $x\in V(D)$, its {\it outer neighbourhood} is defined as 
  $\mathcal{N}_{D}^{+}(x):= \{y \in V(D) | (x,y)\in E(D)\}$, its {\it inner neighbourhood} is defined as $\mathcal{N}_{D}^{-}(x):= \{z \in V(D) | (z,x)\in E(D)\}$, and $\mathcal{N}_D(x):=\mathcal{N}_{D}^{+}(x) \cup \mathcal{N}_{D}^{-}(x)=\mathcal{N}_{G}(x)$. Also, we denote $\mathcal{N}_D[x]:=\mathcal{N}_D(x) \cup \{x\}=\mathcal{N}_{G}[x]$. A vertex $x\in V(D)$ is called a {\it source} if $\mathcal{N}_{D}^{-}(x)=\emptyset$ and $x$ is called a {\it sink} if $ \mathcal{N}_{D}^{+}(x) = \emptyset $. If $x\in V(D)$ is a source, then we may assume $w(x)=1$ as it would not affect the ideal $I(D)$. By $G^c$, we denote the complement of the simple graph $G$. For $A \subseteq V(G)$, we denote $G[A]$ the induced subgraph of $G$ on the set of vertices $A$. By $G\setminus x$, we mean the induced subgraph $G[V(G)\setminus x]$ of $G$. Also, we denote $H(I(D)_{\ev 2})$ as a subgraph (need not be induced subgraph) $H$ of the simple graph $G$ such that $I(H)=I(D)_{\ev 2}$, where $G$ is the underlying simple graph of $D$. \par
  
 Now let us discuss some notations and definitions regarding simple graphs. For a subset $A\subseteq V(G)$, we denote the induced subgraph of $G$ on $A$ by $G[A]$. If $x\in V(G)$ is a vertex, then we write $G\setminus x$ to denote the graph $G[V(G)\setminus \{x\}]$. We use similar notations for weighted oriented graphs also. A graph $G$ is said to be a {\it complete graph} if there is an edge between every pair of vertices of $G$ and a complete graph on $n$ vertices is denoted by $K_{n}$. A {\it clique} of a graph $G$ is a set $A$ of vertices of $G$ such that $G[A]$ is a complete graph. A vertex $x\in V(G)$ is said to be a {\it simplicial vertex} of $G$ if $G[N[x]]$ is a complete graph. A subset $B \subseteq V(G)$ is called an {\it independent set} if no two vertices in $B$ are adjacent in $G$. A set of vertices $C$ of $G$ is called a {\it vertex cover} of $G$ if $C\cap e\neq \emptyset$ for all $e\in E(G)$. A {\it minimal vertex cover} is a vertex cover which is minimal with respect to inclusion. By $G^c$, we denote the complement graph of $G$.

\begin{definition}{\rm
A cycle of length $n$, denoted by $C_n$, is a connected graph having exactly $n$ vertices and $n$ edges. A simple graph $G$ is called {\it chordal} if $G$ has no cycle of length greater than three.
A graph $G$ is called {\it co-chordal} if $G^c$ is chordal. A graph $G$ is said to be {\it bipartite} if it has no induced odd cycle.
}
\end{definition}

\begin{definition}{\rm For any homogeneous ideal $I$ in $R$, the Castelnuovo-Mumford regularity (or simply regularity), denoted by $\reg(I)$, is defined as follows 
 \begin{align*}
 \reg(I) &=  \max \{j - i \mid \beta_{i,j}(I) \neq 0\} \\
         &= \max\{j+i \mid H_{\m}^i(I)_j \neq 0\},   
 \end{align*}
  where $\beta_{i,j}(I)$ is the $(i,j)^{th}$ graded Betti number of $I$ and $H_{\m}^i(I)_j$ denotes the $j^{th}$ graded component of the $i^{t h}$ local cohomology module $H_{\m}^i(I)$. 
  }
\end{definition}
 
\begin{definition}{\rm
   A homogeneous ideal $I$ is $R$ is componentwise linear if for each $d\geq 0$, the ideal generated by $d^{th}$ homogeneous component $I_d$, denoted by $I_{\ev{d}}$, has $d$-linear resolution over $R$, i.e. $\reg(I_{\langle d \rangle})=d$. Note that if $I$ is equigenerated, then linearity and componentwise linearity are equivalent.
   }
\end{definition}
 
\begin{definition}{\rm
    A monomial ideal $I \subseteq R$ has linear quotient property if there exists an order $ u_1 < \cdots < u_m$ on the minimal monomial generating set $\mathcal{G}(I)$ of $I$ such that the colon ideal $((u_1, \ldots, u_{i-1}) : u_i)$ is generated by a subset of the variables, for $i = 2, \ldots, m$.
    }
\end{definition}

Now, let us state the well-known Fr\"{o}berg's theorem \cite{f90}, which has been used frequently in this paper.\medskip

\noindent{\bf Fr\"{o}berg's Theorem.} Let $G$ be a simple graph and $I(G)$ be its edge ideal. Then $I(G)$ has linear resolution if and only if $G^c$ is chordal.
\medskip

For a monomial $u\in R$, define its {\it support} as $\supp(u) := \{x_i : x_i \mid u\}$. 

\begin{definition}\cite[Definition 2.2]{a17} {\rm
A monomial ideal $I$ is said to satisfy the {\em semi-gcd condition }, if for any two monomials $u,v \in \mathcal{G}(I)$ with $\gcd(u,v)=1$ there exists a monomial $w(\neq u,v) \in \mathcal{G}(I)$ such that $\supp(w) \subseteq \supp(u) \cup \supp(v)$.
}
\end{definition}
	
\begin{theorem}\cite[Theorem 2.3]{a17}\label{thm1}
Let $I$ be a monomial ideal, which contains no variable. Assume that there exists an integer $s \geq 1$ such that $I^s$ is componentwise linear. Then $I$ satisfies the semi-gcd condition.  
\end{theorem}

\begin{definition}{\rm
Let $m=x_1^{a_1}\cdots x_n^{a_n}$ be a monomial in $R=\mathbb{K}[x_1,\ldots,x_n]$. Then the {\it polarization} of $m$ is defined to be the squarefree monomial
\begin{equation*}
    \PP(m)=x_{11}x_{12} \cdots x_{1a_1} x_{21} \cdots x_{2a_2} \cdots x_{n a_n}
\end{equation*}
in the polynomial ring $\mathbb{K}[x_{i j} : 1 \leq j \leq a_i,1 \leq i \leq n ].$ If $I \subset R $ is a monomial ideal with $\mathcal{G}(I)=\{m_1, \ldots, m_u\}$ and $m_i=\prod_{j=1}^{n}x_j^{a_{ij}}$ where each $a_{ij} \geq 0$ for $i=1, \ldots, u.$ Then the {\it polarization} of $I$, denoted by $ I^{\PP}$, is defined as 
\begin{equation*}
    I^{\PP}=(\PP(m_1), \ldots, \PP(m_u)),
\end{equation*}
which is a squarefree monomial ideal in the polynomial ring $R^{\PP}=\mathbb{K}[x_{j1},x_{j2}, \ldots, x_{ja_j} \mid j=1, \ldots, n ]$, where $a_j=\max\{a_{i j} \mid i=1, \ldots,u \}$ for any $1 \leq j \leq n.$
}
\end{definition}

\begin{lemma}\cite[Proposition $1$]{nms21}\label{compntpolariz}
   Let $I$ be a monomial ideal in $R=\mathbb{K}[x_1,\ldots,x_n]$. Then $I$ is a componentwise linear ideal if and only if $I^{\PP}$ is a componentwise linear ideal. 
\end{lemma}

\begin{definition}\label{spl1}{\rm 
A monomial ideal $I\subseteq R=\mathbb{K}[X]$ is called {\it vertex splittable} if it can be obtained by the following recursive procedure.
\begin{enumerate}
    \item If $v$ is a monomial and $I=(v)$, $I=(0)$ or $I=R$, then $I$ is vertex splittable.
    \item If there is a variable $x$ in $R$ and vertex splittable ideals $I_1$ and $I_2$ in $K[X\setminus \{x\}]$ such that $I=xI_1+I_2$, $I_2 \subseteq I_1 $ and $\mathcal{G}(I)= \mathcal{G}(xI_1)\sqcup \mathcal{G}(I_2)$, then $I$ is a vertex splittable. For $I=xI_1+I_2$, the variable $x$ is said to be a {\it splitting variable} for $I$.
\end{enumerate}
}
\end{definition}

\begin{lemma}\label{simplicial}
Every chordal graph $G$ has a simplicial vertex. Moreover, if $G$ is a non-complete chordal graph, then it has two nonadjacent simplicial vertices.
\end{lemma}
\begin{proof}
 The proof follows from \cite[Lemma 7]{kmns13}.  
\end{proof}

The following Remark \ref{remadjacent} has been used frequently in our proofs.

\begin{remark}\label{remadjacent}{\rm 
Let $G$ be a simple graph. It is straightforward to verify that if $x$ is a simplicial vertex in $G^c$, then $\mathcal{N}_{G}(x)$ is a minimal vertex cover of $G$. Now, suppose $G$ is a co-chordal graph. Then by Lemma \ref{simplicial}, $G^c$ has a simplicial vertex $x$. Hence, $\mathcal{N}_{G}(x)$ is a minimal vertex cover of $G$. Moreover, if $G^c$ is a non-complete chordal graph, then by Lemma \ref{simplicial}, $G^c$ has at least two non-adjacent simplicial vertices. In other words, if $G$ is a non-empty co-chordal graph, then there exist two adjacent vertices $x$ and $y$ in $G$ such that $\mathcal{N}_{G}(x)$ and $\mathcal{N}_{G}(y)$ both are minimal vertex covers of $G$.
}
\end{remark}
\begin{lemma}\cite[Lemma 2.7]{x21} \label{sum}
Let $S_1=\mathbb{K}[x_1,\ldots,x_m]$ and $S_2=\mathbb{K}[x_{m+1},\ldots,x_n]$ be two polynomial rings, $I \subset S_1$ and $J \subset S_2$ be two nonzero homogeneous ideals. Then
\begin{enumerate}
    \item $\reg(I+J)=\reg(I)+\reg(J)-1,$
    \item $\reg(I J)=\reg(I)+\reg(J).$
\end{enumerate}
\end{lemma}

 \begin{lemma}{\cite[Lemma 4.2]{seyed18}} \label{lemregcolon}
     Let $I\subseteq R$ be a monomial ideal. Then for every monomial $u\in R$, we have $\mathrm{reg}((I:u))\leq \mathrm{reg}(I)$.
 \end{lemma}

\begin{lemma}\label{lemreginduced}
    Let $I$ be a monomial ideal of $R$. Then for every variable $x$ appearing in $\mathcal{G}(I)$, we have $\mathrm{reg}(I,x)\leq \mathrm{reg}(I)$.
\end{lemma}
\begin{proof}
    Follows from the proof of \cite[Lemma 2.10]{dhs13}.
\end{proof}
\begin{lemma}\cite[Corollary 8.2.14]{hh11}\label{maximum}
 Let $I \subseteq R=\mathbb{K}[x_1,x_2,\ldots,x_n]$ be a componentwise linear ideal. Then the regularity of $I$ is equal to the highest degree of a generator in a minimal set of generators of $I$.   
\end{lemma}

\section{Necessary conditions for componentwise linearity}\label{seccochordal}

In this section, we study some necessary conditions for an edge ideal of a weighted oriented graph to be componentwise linear. Our main result says that for any weighted oriented graph $D$, if $I(D)$ is componentwise linear, then the underlying simple graph $G$ of $D$ is co-chordal, i.e., $\sqrt{I(D)}=I(G)$ (the radical of $I(D)$) has linear resolution which is not the case in general. 

\begin{proposition}\label{proplin}
		Let $I\subseteq R$ be a monomial ideal of degree $d$ with linear resolution. For a variable $x$ in $R$, consider the ideal $I'$ such that $\mathcal{G}(I')=\mathcal{G}(I) \setminus \{u \mid u \in \mathcal{G}(I), x \mid u\}$. Then $I'$ also has a linear resolution.   
	\end{proposition}
	\begin{proof}
		Given that $I$ is a monomial ideal of degree $d$ with linear resolution. Thus, $\reg(I)=d$. Also, observe that $I' $ is either $0$ or a monomial ideal of degree $d$. If $I'=0$, then it has linear resolution with $\reg(I')=0$. If $I'\neq 0$, then $\reg(I') \geq d$ as $\mathcal{G}(I')$ contains a generator of degree $d$. Again, by Lemma \ref{lemreginduced}, we have
		\begin{align*}
			\reg(I')=\reg(I',x)=\reg(I,x)\leq \reg(I)=d  .
		\end{align*}
		Therefore, $I'$ has a linear resolution.
	\end{proof}
 
\begin{proposition}\label{propcl}
Let $I\subseteq R$ be a componentwise linear monomial ideal. For a variable $x$ in $R$, consider the ideal $I'$ such that $\mathcal{G}(I')=\mathcal{G}(I) \setminus \{u \mid u \in \mathcal{G}(I), x \mid u\}$. Then $I'$ is also componentwise linear.     
\end{proposition}

	\begin{proof}
		Let $I_{\langle d \rangle}$ be the $d$th homogeneous component of $I$, where $d$ is a positive integer. By looking at the minimal generating set of $I'$, it is easy to observe that $(I'_{\langle d \rangle},x)=(I_{\langle d \rangle},x)$. Therefore, by Preposition \ref{proplin}, we have $I'_{\langle d \rangle}$ has a linear resolution as $I_{\langle d \rangle}$ has a linear resolution. Thus, $I'$ is componentwise linear as $d$ was chosen arbitrarily.
	\end{proof}
	
	\noindent

 \begin{figure}[!h]
    \centering
     \includegraphics[width=1.0 \textwidth]{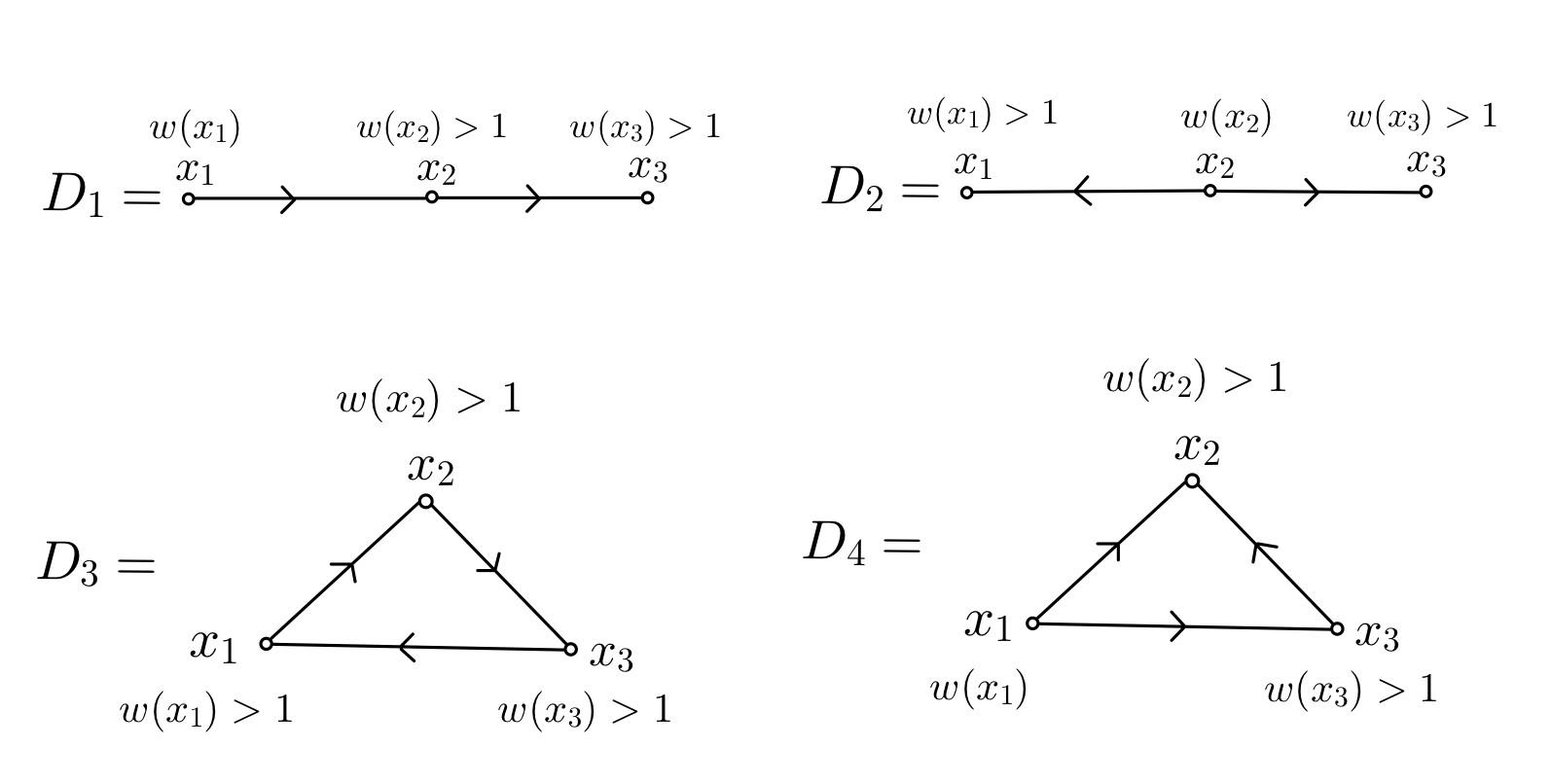} 
    \caption{$D_1, D_2, D_3, D_4$ can not be an induced subgraph of a weighted oriented graph $D$ whose edge ideal is componentwise linear.}
    \label{fig1}
\end{figure}
\begin{proposition}\label{notmaximum}
 Let $D_1, D_2, D_3$, and $D_4$ be weighted oriented graphs as in Figure \ref{fig1}. Then $\reg(I(D_i)) > d_i + 1 $, where $d_i=\max\{w(x) \mid x \in V(D_i)\}$ for $i=1,2,3,4$.
\end{proposition}
\begin{proof} We will show the required result in four individual cases: \\
{\bf Case 1:} Suppose $D_1$ is a weighted oriented graph as in Figure \ref{fig1}. Since $w(x_2)>1$ and $w(x_3)>1$, by \cite[Lemma 3.4]{x21},  we get
\begin{equation*}
 \reg(I(D_1))=\displaystyle\sum_{i=1}^3w(x_i)-1 > d_1+1.   
\end{equation*}
\noindent 
 {\bf Case 2:} Suppose $D_2$ is a weighted oriented graph as in Figure \ref{fig1}. Then
 \begin{align*}
 \reg(I(D_2))&=\reg((x_2x_1^{w(x_1)},x_2x_3^{w(x_3)})) \\
             &=\reg((x_2(x_1^{w(x_1)},x_3^{w(x_3)})))  \\
            &=\reg((x_1^{w(x_1)},x_3^{w(x_3)}))+1 \mbox{ ( by Lemma \ref{sum}(2) ) } \\
            &=w(x_1)+w(x_2).
 \end{align*}
 Thus, by $w(x_1)>1$ and $w(x_3)>1$ we have,
 \begin{equation*}
    \reg(I(D_2))=w(x_1)+w(x_2) > d_2+1.  
 \end{equation*}
\noindent
 {\bf Case 3:} Suppose $D_3$ is a naturally weighted oriented graph as in Figure \ref{fig1}. Then from \cite[Theorem 1.4]{z19} and $w(x_i)>1$ for all $i=1,2,3$, it follows that
 \begin{equation*}\label{eq1}
    \reg(I(D_3))=\sum_{i=1}^3w(x_i)-\vert E(D_3) \vert+1=\sum_{i=1}^3w(x_i)-2>d_3+1. 
 \end{equation*}
\noindent
  {\bf Case 4:} Suppose $D_4$ is a weighted oriented graph as in Figure \ref{fig1}. Since $D_4$ is not naturally oriented and $V^{+}$ are sinks, in $D_4$, one vertex is always a sink vertex, and one vertex is always a source vertex. In the figure, we see that $x_1$ is a source vertex and $x_2$ is a sink vertex. Now, take the following exact sequence
 \begin{equation}\label{eqn1}
     0 \rightarrow (I(D_4): x_2^{w(x_2)})(-w(x_2)) \xrightarrow{.x_2^{w(x_2)}} I(D_4) \rightarrow (I(D_4),x_2^{w(x_2)}) \rightarrow 0
 \end{equation}
 Note that 
 \begin{align*}
    \reg((I(D_4): x_2^{w(x_2)})(-w(x_2))&=\reg(I(D_4): x_2^{w(x_2)})+w(x_2) \\
                                        &=\reg((x_1,x_3))+w(x_2)=w(x_2)+1.
 \end{align*}
  Again, we have $\reg((I(D_4),x_2^{w(x_2)}))=\reg((I(D_4\setminus x_2),x_2^{w(x_2)}))=w(x_2)+w(x_3)$. Therefore, by applying regularity lemma \cite[Corollary 20.19]{e95} on the exact sequence \eqref{eqn1}, we have $\reg(I(D))=w(x_2)+w(x_3) > d_4+1$.
\end{proof}

\begin{corollary}\label{inducedcomponent}
    Let $D$ be a weighted oriented graph. Suppose $I(D)$ is componentwise linear. Then 
 $D_1,D_2,D_3$ and $D_4$ as in Figure \ref{fig1} cannot be an induced weighted oriented subgraphs of $D$.
\end{corollary}

\begin{proof}
 Note that by Lemma \ref{maximum} and Proposition \ref{notmaximum}, $I(D_i)$ is not componentwise linear for $i=1,2,3,4$. Therefore by Lemma \ref{propcl}, $I(D)$ is not componentwise linear which is a contradiction.   
\end{proof}

\begin{corollary}\label{notcomponent}
 Let $D$ be a weighted oriented graph such that $D$ is not a weighted oriented star graph with all the edges towards the root of the weighted oriented star graph. Assume all the weights of non-source vertices are greater than one. Then $I(D)$ is not componentwise linear.     
\end{corollary} 
\begin{proof}
Note that $D$ has a possible induced subgraph on three vertices $x_1,x_2,x_3$, which are isomorphic to one of $D_1, D_2, D_3$ and $D_4$ as in Figure \ref{fig1}. Thus, by Corollary \ref{inducedcomponent}, $I(D)$ is not componentwise linear.    
\end{proof}

 \begin{lemma}\label{lemcochord}
Let $D=(V(D), E(D), w)$ be a weighted oriented graph with underlying graph $G$ such that $G^c$ is a cycle of length greater than $3$ and $V^{+}$ are sinks. Then $\mathrm{reg}(I(D))>d+1$, where $d=\mathrm{max}\{w(x)\mid x\in V(D)\}$.
 \end{lemma}

 \begin{proof}
If $d=1$, then $I(D)=I(G)$ and the assertion holds by Fr\"{o}berg's theorem. Therefore, we assume $d>1$. Let $V(D)=\{x_{1},\ldots,x_{n}\}$ such that $E(G^c)=\{\{x_{i},x_{i+1}\}\mid 1\leq i\leq n\,\,\text{and}\,\, x_{n+1}=x_1\}$. Let $x_{k}$ be the vertex such that $w(x_k)=d$. From our labelling, it is clear that $\mathcal{N}_{G}(x_i)=V(G)\setminus\{x_{i-1},x_{i+1}\}$ because $\mathcal{N}_{G^c}(x_i)=\{x_{i-1},x_{i+1}\}$ and $\{x_{i-1},x_{i+1}\}\in E(G)$. Suppose $w(x_{k+1})>1$. Take a vertex $x\in V(G)\setminus \{x_{k-1},x_{k},x_{k+1},x_{k+2}\}$. Then we have $(x,x_{k}),(x,x_{k+1})\in E(D)$ as $V^{+}$ are sinks. Thus, we can write 
$$(I(D):x)=(x_{k}^{d},x_{k+1}^{w(x_{k+1})})+(\mathcal{N}_{D}(x)\setminus \{x_{k},x_{k+1}\})+I(D\setminus \mathcal{N}_{D}[x]).$$ 
Note that $I(D\setminus \mathcal{N}_{D}[x])=I(G\setminus \mathcal{N}_{G}[x])=I(K_2)$ as $V^{+}$ are sinks. Therefore, by Lemma \ref{sum}, we have $\mathrm{reg}(I(D):x)=d+w(x_{k+1})-1+\mathrm{reg}(I(D\setminus \mathcal{N}_{D}[x]))-1=d+w(x_{k+1})>d+1$. Hence, by Lemma \ref{lemregcolon}, we get $\reg(I(D))>d+1$. The case $w(x_{k-1})>1$ is similar. Now, suppose $w(x_{k-1})=w(x_{k+1})=1$. Observe that $I(D): x_k=I(D')$, where $D'=(v(D'),E(D'),w')$ is a weighted oriented graph with underlying graph and orientation same as $D$, but, $w'(x_k)=d-1$ and $w'(x_{i})=w(x_i)$ for all $i\neq k$. Then by induction hypothesis, we have $\mathrm{reg}(I(D):x_k)> d-1+1=d$. Again, due to our assumption $w(x_{k-1})=w(x_{k+1})=1$ and $V^{+}$ being sinks, we have $I(D\setminus x_k)=I(G\setminus x_k)$. Since $G^c$ is a cycle, $(G\setminus x_k)^c$ is a path graph and so, is chordal. Therefore, we have $\mathrm{reg}(I(D\setminus x_k))=2$ by Fr\"{o}berg's theorem. Thus, $\reg((I(D),x_k))=2\neq \mathrm{reg}(I(D))$ as $d>1$. Also, by \cite[Corollary 3.3(2)]{chhktt19}, $\mathrm{reg}(I(D))\in \{\mathrm{reg}(I(D),x_k),\mathrm{reg}(I(D):x_k)+1\}$ and hence, we get $\reg(I(D))>d+1$, as required. 
 \end{proof}


\begin{proposition}\label{thmind4}
	Let $D$ be a weighted oriented graph with underlying graph $G$. If $I(D)^k$ is componentwise linear for some $k \geq 1$, then $G^c$ has no induced $4$-cycle.  
	\end{proposition}
	\begin{proof}
		If possible, suppose $G^c$ has an induced $4$-cycle with vertices $x_1,x_2,x_3,x_4$. Take $u=x_1x_3^{w(x_2)}$ or $x_1^{w(x_1)}x_3$ and $v=x_2x_4^{w(x_4)}$ or $x_2^{w(x_2)}x_4$. Then $u, v  \in \mathcal{G}(I(D))$ and $gcd(u,v)=1$. But there is no monomial $w\neq u,v \in \mathcal{G}(I(D))$ with $\supp(w) \subseteq \supp(u) \cup \supp(v)$. Thus, by Theorem \ref{thm1}, we have no power of $I(D)$ is componentwise linearity, which contradicts our assumption. Therefore, $G^c$ has no induced $4$-cycle.  
	\end{proof}
 
\begin{theorem}\label{thmcochordal}
	Let $D=(V(D), E(D), w)$ be a weighted oriented graph with the underlying simple graph $G$. If $I(D)$ is componentwise linear, then $G^c$ is chordal.   
\end{theorem}
	\begin{proof}
		Let $\displaystyle m=\sum_{x \in V(D)} w(x)$. We proceed by induction on $m$. Note that $m\geq n$, where $n$ is the number of vertices of $D$. So, the base case will be $m=n$. Note that $m=n$ if and only if $w(x)=1$ for all $x\in V(D)$. In this case, we have $I(D)=I(G)$ and thus, $I(D)$ is componentwise linear implies $I(D)$ has a linear resolution as $I(D)$ is equigenerated in degree $2$. Hence, $G^c$ is chordal by Fr\"{o}berg's theorem. Now let us assume $m>n$. Then there exists at least one vertex $x_i$ with $w(x_i)>1$. From Proposition \ref{thmind4}, it follows that $G^c$ has no induced cycle of length $4$. Now suppose $G^c$ is not chordal, and it contains a cycle $C$ of length greater than or equal to $5$. Let $D_{C}$ be the induced weighted oriented subgraph of $D$ on the vertex set $V(C)$, i.e., $D_{C}=D[V(C)]$. Now, we will consider two possible cases:\par 
  
		\noindent {\bf \underline{Case 1:}} Suppose $D$ contains a vertex $x_j \not \in V(D_c)$. Then $I(D\setminus x_j)$ is componentwise linear by Proposition \ref{propcl}. Note that the underlying simple graph of $D\setminus x_{j}$ is $G\setminus x_j$. Since $x_j\not\in V(C^c)$, $C^c$ is an induced subgraph of $G\setminus x_j$. By induction hypothesis, $(G\setminus x_{j})^c$ is chordal, which contradicts the fact that $C^c$ is an induced subgraph of $G\setminus x_j$. The complement of $D\setminus x_j$ contains the cycle $C$ which is a contradiction by induction hypothesis.\par
  
	\noindent {\bf \underline{Case 2:}} Suppose $V(D)=V(D_c)$. Since $I(D)$ is componentwise linear, by Lemma \ref{maximum}, we have $\mathrm{reg}(I(D))=d+1$, where $d=\mathrm{max}\{w(x)\mid x\in V(D)\}$. Then by Lemma \ref{lemcochord}, all vertices of $V^{+}$ can not be sink as $G^c$ is a cycle of length greater than or equal to $5$. Let $V(G^c)=\{x_1,\ldots,x_n\}$ and $E(G^c)=\{\{x_{i},x_{i+1}\}\mid 1\leq i\leq n\,\,\text{and}\,\, x_{n+1}=x_{1}\}$. Now there exists a vertex in $V^+$, which is not a sink vertex. Without loss of generality, let that vertex be $x_1$ and $w(x_1)=k>1$. Since $I(D)$ is componentwise linear, by Lemma \ref{compntpolariz}, we have $I(D)^{\PP}$ is also componentwise linear. From the structure of $D$, it is clear that $\mathcal{N}_{D}(x_1)=\{x_3,\ldots,x_{n-1}\}$. Let $\mathcal{N}_{D}^{-}(x_1)=\{x_{i_1},\ldots,x_{i_{r}}\}$. Now consider the weighted oriented graph $D'=(V(D'),E(D'),w')$ as follows:
 \begin{center}
     $V(D')=V(D),\,\, E(D')=E(D)\setminus\{(x_{i_{1}},x_{1}),\ldots, (x_{i_r},x_1)\},$\\ 
     $w'(x_1)=1,\,\, w'(x_j)=w(x_j)$ for all $2\leq j\leq n$.
 \end{center}
Now we can write $(I(D)^{\PP},x_{1k})=(I',x_{1k})$, where $\mathcal{G}(I')=\{u \in \mathcal{G}(I(D)^{\PP}) \mid x_{1k} \nmid u\}$. Then from the construction of $D'$, it is easy to observe that $I'=I(D')^{\PP}$. Thus, by Proposition \ref{propcl}, we have $I(D')^{\PP}$ is componentwise linear and hence, $I(D')$ is componentwise linear by Lemma \ref{compntpolariz}. Let $G'$ be the underlying graph of $D'$. If we look at the complement of $G'$, then we can see that $E((G')^{c})=E(G^c)\cup \{\{x_1,x_{i_{1}}\},\ldots,\{x_{1},x_{i_{r}}\}\}$. Since $x_1$ is not a sink vertex in $D$, $\mathcal{N}_{D}^{+}(x_1)\neq\emptyset$. Thus, in $(G')^c$, there should exists a vertex $x_{i}$ such that $\{x_1,x_i\}\not\in E((G')^c)$. Therefore, $(G')^c$ contains an induced cycle of length greater than or equal to $4$ as $G^c$ is a cycle of length greater than or equal to $5$. Hence, $G'$ is not co-chordal. Since $w'(x_1)=1$ and $I(D')$ is componentwise linear, by induction hypothesis, $G'$ should be co-chordal and this gives a contradiction.
In both cases, we came up with a contradiction.
\end{proof}

The example below conveys that a monomial ideal may be componentwise linear, whereas its radical may not be so. But, from our Theorem \ref{thmcochordal}, we see that if $I(D)$ is componentwise linear, then $\sqrt{I(D)}=I(G)$ has a linear resolution. 

  \begin{example}\label{radical}{\rm
   Let $I \subset R=\mathbb{Q}[x_1,x_2,x_3,x_4,x_5]$ be a monomial ideal such that $I=(x_1x_2,x_2x_3,x_3x_4,x_4x_5,x_5x_1)$ whcih is the edge ideal of a $5$-cycle. By Fr\"oberg's theorem, $I$ has no linear resolution. Let $J=I^2$. Note that $\sqrt{J}=I$. Using Macaulay2 \cite{gs} we check that, $\reg(J)=4$. Thus, $J$ has linear resolution.
   }
  \end{example}
 
\section{Combinatorial characterization of componentwise linearity}\label{secvertexsplit}

In this section, we investigate some sufficient conditions for componentwise linearity of weighted oriented edge ideals. We mainly use the tools of vertex splitting and linear quotient property. In some cases, we give combinatorial characterizations of componentwise linearity for these ideals. Due to our results and computational evidence, we end up this section with Conjecture \ref{conjcl}.

\begin{lemma}\label{lemspl3}
Let $I \subseteq R$ be a monomial ideal. Then the following are true.
\begin{enumerate}
\item If $I=(I',x_{i_1},\ldots,x_{i_r})$, where $I'$ is vertex splittable and no $x_{i_{j}}$ is appearing in $\mathcal{G}(I')$, then $I$ is vertex splittable ideal.
\item If $I=(x_{i_1}u,\ldots,x_{i_{r}}u)$ with $u$ as a monomial, then $I$ is vertex splittable. In particular, if $I(D)$ has linear resolution, then $I(D)$ is vertex splittable.
\end{enumerate}
\end{lemma}

\begin{proof}
$(1).$ Taking $x_{i_1}$ as a splitting variable, we can write $I$ as follows
$$ I=x_{i_1}R+(I',x_{i_2},\ldots,x_{i_r}).$$
Note that $(I',x_{i_2},\ldots,x_{i_r}) \subseteq R$. Now, using induction on $r$ (where the base case is a routine check), we get $(I',x_{i_2},\ldots,x_{i_r})$ is vertex splittable. Hence, $I$ is vertex splittable. \\

$(2).$ In this case, if $r=1$, then $I$ is generated by one monomial and thus, $I$ is vertex splittable by Definition \ref{spl1}. For $r>1$, we can write $I$ as follows
$$I=x_{i_1}(u)+(x_{i_2}u,\ldots,x_{i_r}u).$$
Clearly, we have $(x_{i_2}u,\ldots,x_{i_r}u)\subseteq (u)$. Again, inductively, we have $(x_{i_2}u,\ldots,x_{i_r}u)$ is vertex splittable. Therefore, $I$ is vertex splittable by Definition \ref{spl1}. Let $D$ be a weighted oriented graph with the underlying graph $G$. If $I(D)=I(G)$ and $I(D)$ has a linear resolution, then it is vertex splittable by \cite[Corollary 3.8]{mka16}. From \cite[Theorem 3.5]{bds23}, it follows that if $I(D)\neq I(G)$ and $I(D)$ has a linear resolution, then $I(D)$ has the form as in $(2)$ and hence, $I(D)$ is vertex splittable.
\end{proof}

It has been seen in the literature that $I(D)$ behaves well, sometimes like $I(G)$, when $V^{+}$ are sinks. Take graph $D_2$ as in Figure \ref{fig1}. By Corollary \ref{notcomponent}, $I(D)$ is not componentwise linear even if the underlying graph of $D_2$ is co-chordal and $V^+$ are sinks. In the following theorem, we fully characterize componentwise linear weighted oriented edge ideals when $V^{+}$ are sinks.

 \begin{theorem}\label{thmsink}
     Let $D=(V(D), E(D), w)$ be a weighted oriented graph with the underlying graph $G$ such that $V^{+}$ are sinks. Then the following are equivalent.
     \begin{enumerate}
         \item $I(D)$ is componentwise linear;
         \item $I(D)$ is vertex splittable;
         \item $I(D)$ has linear quotient property;
         \item $G^c$ is chordal and $\vert\mathcal{N}_{D}(x)\cap V^{+}\vert\leq 1$ for all $x\in V(D)$.
    \end{enumerate}
 \end{theorem}

 \begin{proof}
 $(2)\implies (3)\implies (1)$ is known. Therefore, it is enough to show $(1)\implies (4)$ and $(4)\implies (2)$.\par 

 $(1)\implies (4)$: Let $I(D)$ be componentwise linear. Then $G^c$ is chordal by Theorem \ref{thmcochordal}. Suppose there exists a vertex $x\in V(D)$ such that $\vert\mathcal{N}_{D}(x)\cap V^{+}\vert >1$. Take two vertex $x_1$ and $x_2$ from $\mathcal{N}_{D}(x)\cap V^{+}$. Now, consider the induced weighted subgraph $D'=D[\{x,x_1,x_2\}]$. Since $x_1$ and $x_2$ are sink vertices, we have $E(D')=\{(x,x_1), (x,x_2)\}$ which is same as $D_2$ in Figure \ref{fig1}. But, by Corollary \ref{inducedcomponent}, $D_2$ cannot be an induced subgraph of $D$. Hence, $\vert\mathcal{N}_{D}(x)\cap V^{+}\vert \leq 1$.\par 
 
$(4)\implies (2)$: We may assume $G^c$ is not a complete graph; otherwise, $G$ will be empty. Since $G^c$ is chordal, there exists a simplicial vertex of $G^c$, say $x$. Then $\mathcal{N}_{G}(x)$ is a minimal vertex cover of $G$. Since $G^c$ is not a complete graph, by Remark \ref{remadjacent}, there exists $x'\in\mathcal{N}_{G}(x)$ such that $\mathcal{N}_{G}(x')$ is a minimal vertex cover of $G$. Since $V^{+}$ are sinks, at least one of $x$ or $x'$ has weight $1$. Without loss of generality, let us assume $w(x)=1$. By the given condition, we have $\vert\mathcal{N}_{D}(x)\cap V^{+}\vert\leq 1$. Let $\mathcal{N}_{D}(x)=\{y_1,\ldots,y_k\}$ such that $w(y_1)=\cdots=w(y_{k-1})=1$ and $w(y_k)$ may be any positive integer. Therefore, we can write $I(D)$ as follows:
     $$I(D)=x(y_1,\ldots,y_{k-1},y_{k}^{w(y_k)})+ I(D\setminus x).$$
     Now, by Lemma \ref{lemspl3}, the ideal $(y_1,\ldots,y_{k-1},y_{k}^{w(y_k)})$ is vertex splittable. Note that $G^c$ chordal implies $(G\setminus x)^c$ is again chordal and $\vert\mathcal{N}_{D\setminus x}(z)\cap V(D\setminus x)^{+}\vert\leq 1$ for all $z\in V(D\setminus x)$ is straight forward from the structure of $D$. Thus, applying induction on the number of vertices of $D$ (where the base case follows directly from Definition \ref{spl1}), we have $I(D\setminus x)$ is vertex splittable. Again, $(y_1,\ldots,y_k)$ being a minimal vertex cover of $G$ and $y_k$ being a sink vertex of $D$, we have $$I(D\setminus x)\subseteq (y_1,\ldots,y_{k-1},y_{k}^{w(y_k)}).$$
         Hence, by Definition \ref{spl1}, it follows that $I(D)$ is vertex splittable.
 \end{proof}
 
 \begin{figure}[!h]
    \centering
     \includegraphics[width=0.38 \textwidth]{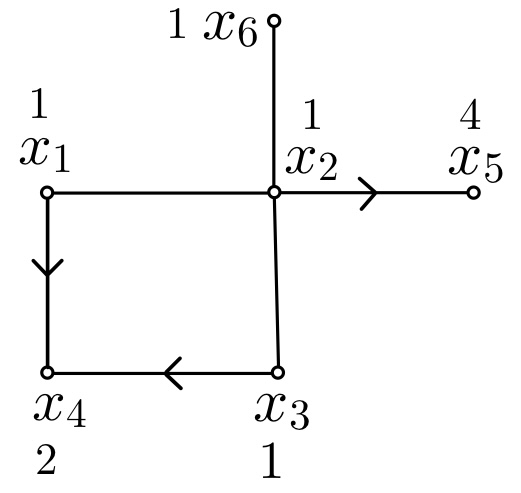} 
    \caption{Graph $D$ with $V^{+}$ sinks and componentwise linear $I(D)$.}
    \label{fig2}
\end{figure}
\begin{example}\label{exsinknotcl}
Let $D$ be a weighted oriented graph as in Figure \ref{fig2}. Then the edge ideal of $D$ is $I(D)=(x_1x_2,x_2x_6,x_2x_3,x_2x_5^{4},x_1x_4^2,x_3x_4^2)$. Note that the underlying simple graph of $D$ is co-chordal and $V^{+}$ are sinks. Also, $\vert\mathcal{N}_{D}(x)\cap V^{+}\vert\leq 1$ for all $x\in V(D)$. Thus, by Theorem \ref{thmsink}, $I(D)$ is componentwise linear.
\end{example}

Next, we will classify all componentwise linear edge ideals of weighted oriented graphs when the number of vertices having non-trivial weights (i.e. weight greater than $1$) is $1$.

 \begin{theorem}\label{thmcardinality}
     Let $D=(V(D), E(D), w)$ be a weighted oriented graph with the underlying graph $G$ and $\vert V^{+}\vert\leq 1$. Then the following statements are equivalent.
     \begin{enumerate}
         \item $I(D)$ is componentwise linear;
         \item $I(D)$ is vertex splittable;
         \item $I(D)$ has linear quotient property;
         \item $G$ and $H=H(I(D)_{\ev{2}})$ both are co-chordal.
     \end{enumerate}
 \end{theorem}

 \begin{proof}
 $(2) \implies (3) \implies (1) \implies (4)$ is clear.
 It is enough to show only $(4)\implies (2)$. If $\vert V^{+}\vert=0$, then $I(D)=I(G)$ and the result holds by Fr\"{o}berg's theorem and by \cite[Corollary 3.8]{mka16}. Thus, we assume $\vert V^{+}\vert=1$. \par 

$(4)\implies (2):$ Given that $G^c$ is chordal. We may assume $G^c$ is not complete; otherwise, $G$ will be an empty graph. Then by Dirac's theorem, there exist two non-adjacent simplicial vertices in $G^c$, and hence, there exist two adjacent vertices $x$ and $x'$ in $G$ such that $\mathcal{N}_{G}(x)$ and $\mathcal{N}_{G}(x')$ are minimal vertex covers of $G$. Since $\vert V^{+}\vert=1$, without loss of generality we assume $w(x)=1$. Let $\mathcal{N}_{D}^{+}(x)=\{x_1,\ldots,x_r\}$ and $\mathcal{N}_{D}^{-}(x)=\{y_1,\ldots,y_s\}$. If $w(x_i)=1$ for all $1\leq i\leq r$, then we can write
$$I(D)=x(x_1,\ldots,x_r,y_1,\ldots,y_s)+I(D\setminus x).$$
Since $G$ and $H$ bot are co-chordal, $G\setminus x$ and $H\setminus x$ are also co-chordal. Also, it is easy to verify that $H\setminus x=H(I(D\setminus x)_{\ev{2}})$. Since $\mathcal{N}_{G}(x)$ is a minimal vertex cover of $G$, we have $I(D\setminus x)\subseteq (x_1,\ldots,x_r,y_1,\ldots,y_s)$. Hence, by induction on the number of vertices of $D$ (the base case is easy to verify) and Definition \ref{spl1}, we get $I(D)$ is vertex splittable. Now, without loss of generality, we assume $w(x_1)>1$. In this set-up, if there exists no $z\in \mathcal{N}_{D}^{+}(x_1)\setminus \mathcal{N}_{D}(x)$ such that $(x_1,z)\in E(D)$, then we have
$I(D)=x(x_{1}^{w(x_1)},x_2,\ldots,x_r,y_1,\ldots,y_s)+I(D\setminus x),$
where $I(D\setminus x)\subseteq (x_{1}^{w(x_1)},x_2,\ldots,x_r,y_1,\ldots,y_s)$. By the same argument as above, it follows that $I(D)$ is vertex splittable. Now, suppose there exists $z\in \mathcal{N}_{D}^{+}(x_1)\setminus \mathcal{N}_{D}(x)$. We consider the graph $H=H(I(D)_{\ev 2})$. Given that $H^c$ is chordal. In this case, if $\{x_2,\ldots,x_r,y_1,\ldots,y_s\}\subseteq \mathcal{N}_{H}(x_1)$, then writing $I(D)=x_1(\{ux_{1}^{w(x_1)-1}\mid u\in \mathcal{N}_{D}^{-}(x_1)\}, \mathcal{N}_{D}^{+}(x_1))+I(D\setminus x_1)$ and using induction hypothesis, we can easily verify that $I(D)$ is vertex splittable. Suppose $\{x_2,\ldots,x_r,y_1,\ldots,y_s\}\not\subseteq \mathcal{N}_{H}(x_1)$ and take $v\in \{x_2,\ldots,x_r,y_1,\ldots,y_s\}\setminus \mathcal{N}_{H}(x_1)$. Note that $\{x,x_{1}\},\{x,z\}\not\in E(H)$. Since $H^c$ is chordal, $\{x_1,z\}$ and $\{x,v\}$ can not form an induced matching in $H$ and thus, we should have $\{v,z\}\in E(H)$. Let us take any arbitrary element $v'\in \{x_2,\ldots,x_r,y_1,\ldots,y_s\}\setminus \{v\}$. Now observe that $\{x_1,z\}, \{x,v'\}$ can not form an induced matching in $H$ and $H[\{x,v,z,x_1,v'\}]$ can not be an induced $5$-cycle in $H$ as $H^c$ is chordal. Therefore, the only possibility is $\{v',z\}\in E(H)$. Thus, $\mathcal{N}_{H}(z)=\{x_1,\ldots,x_r,y_1,\ldots,y_s\}$ and we can write
$$I(D)=z(x_1,\ldots,x_r,y_1,\ldots,y_s)+I(D\setminus z),$$
where $I(D\setminus z)\subseteq (x_1,\ldots,x_r,y_1,\ldots,y_s)$ as $N_{G}(x)=N_{G}(z)$ is a minimal vertex cover of $G$. Hence, applying induction on the number of vertices and using Definition \ref{spl1}, we get $I(D)$ is vertex splittable.
\end{proof}
 \begin{figure}[!h]
    \centering
     \includegraphics[width=1.0 \textwidth]{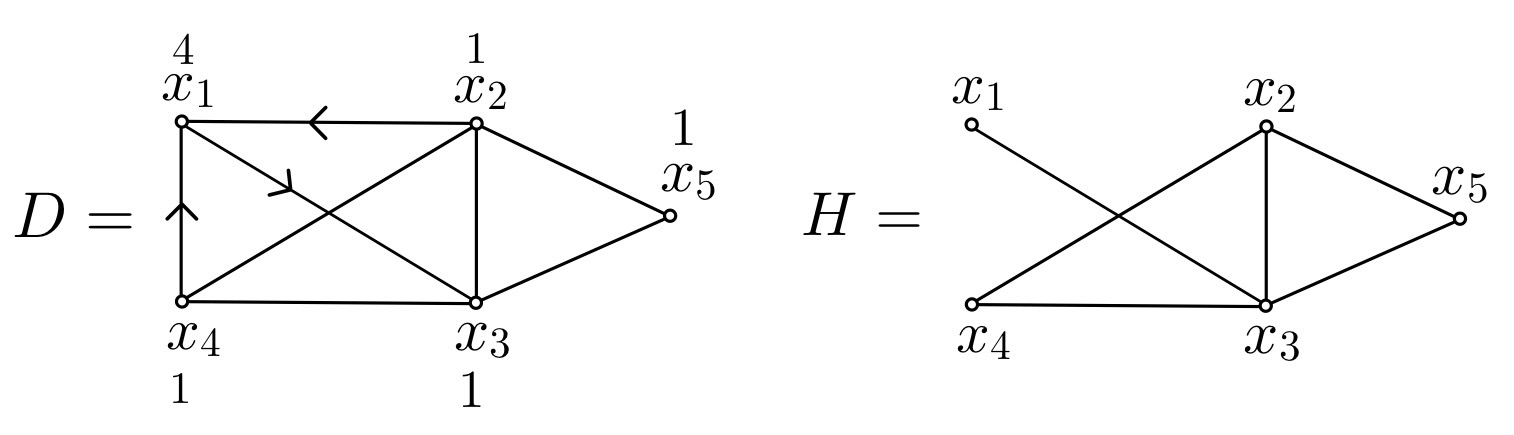} 
    \caption{Weighted oriented graph $D$ with $\vert V^{+}\vert=1$ such that $I(D)$ is componentwise linear, and $H$ is the simple graph $H(I(D)_{\ev{2}})$.}
    \label{fig3}
  \end{figure}
\begin{example}\label{one}
Let $D$ be a weighted oriented graph as in Figure \ref{fig3}. Then the edge ideal of $D$ is $I(D)=(x_4x_1^4, x_2x_1^4, x_1x_3, x_4x_3, x_2x_3, x_3x_5, x_2x_5, x_4x_2)$. Observe that the underlying simple graph of $D$ is co-chordal and $\vert V^{+}\vert = 1$. Also, the graph $H$ in Figure \ref{fig3} is co-chordal and $H=H(I(D)_{\ev 2})$. Thus, by Theorem \ref{thmcardinality}, we have $I(D)$ is componentwise linear.
\end{example}

 \begin{theorem}\label{thmbipartite}
     Let $D=(V(D), E(D), w)$ be a weighted oriented bipartite graph. Then the following statements are equivalent.
     \begin{enumerate}
         \item $I(D)$ is componentwise linear;
         \item $I(D)$ is vertex splittable;
         \item $I(D)$ has linear quotient property.
     \end{enumerate}
 \end{theorem}

 \begin{proof}
  $(2)\implies (3)\implies (1)$ is clear. Thus, it is enough to show $(1)\implies (2)$. \par
  \noindent
  $(1) \implies (2):$
     Let $I(D)$ be componentwise linear and $G$ be the underlying simple graph of $D$. We proceed by induction on the number of vertices $n$ of $D$. The cases $n=1,2$ directly follow from the definition of vertex splittable ideal. Let $n>2$ and $D$ be non-empty. By Theorem \ref{thmcochordal}, we have $G^c$ is chordal. We may assume $G^c$ is not a complete graph as $G$ is non-empty. Then by Lemma \ref{simplicial}, we have two non-adjacent simplicial vertices in $G^c$. In other words, we have two adjacent vertices $x$ and $x'$ in $G$ such that both $\mathcal{N}_{G}(x)$ and $\mathcal{N}_{G}(x')$ are minimal vertex covers of $G$. Note that if $u\in V^{+}$ is a vertex such that $\mathcal{N}_{G}(u)$ is a minimal vertex cover of $G$, then $w(v)=1$ for all $v\in \mathcal{N}_{G}(u)$ because $G$ is a bipartite graph, and $D_1$ or $D_2$ can not be an induced weighted oriented subgraph of a componentwise linear ideal by Corollary \ref{inducedcomponent}. Therefore, we have either $w(x)=1$ or $w(x')=1$ as $\{x,x'\}\in E(G)$. Without loss of generality, we assume $w(x)=1$. Let $\mathcal{N}_{D}^{+}(x)=\{x_1,\ldots,x_r\}$ and $\mathcal{N}_{D}^{-}(x)=\{y_1,\ldots,y_s\}$. Since $D_2$ can not be an induced weighted oriented subgraph of $D$ by Corollary \ref{inducedcomponent} and $G$ is a bipartite graph, we can say that at least $r-1$ vertices of $\mathcal{N}_{D}^{+}(x)$ will have weight $1$. Without loss of generality, we assume $w(x_1)\geq 1$ and $w(x_2)=\cdots=w(x_r)=1$. Now, we will consider two possible cases.\par 
     \noindent \textbf{\underline{Case-1:}} When $w(x_1)=1$ or $x_1$ is a sink vertex. In this case, we can write $I(D)$ as follows:
     $$I(D)=x(x_{1}^{w(x_1)},x_2,\ldots,x_r,y_1,\ldots,y_s)+I(D\setminus x).$$
     Since $w(x_1)=1$ or $x_1$ is sink, $\mathcal{N}_{G}(x)$ is a minimal vertex cover of $G$ implies $I(D\setminus x)\subseteq (x_{1}^{w(x_1)},x_2,\ldots,x_r,y_1,\ldots,y_s)$. Now $I(D\setminus x)$ is componentwise linear by Proposition \ref{propcl}. Thus, by the induction hypothesis, $I(D\setminus x)$ is vertex splittable. Also, by Lemma \ref{lemspl3}, $(x_{1}^{w(x_1)},x_2,\ldots,x_r,y_1,\ldots,y_s)$ is vertex splittable. Hence, $I(D)$ is vertex splittable again from Definition \ref{spl1}.\par 

     \noindent \textbf{\underline{Case-2:}} When $w(x_1)>1$ and $x_1$ is not a sink vertex. Since $G$ is bipartite, no neighbor of $x_1$ is adjacent to $x$. Note that $\mathcal{N}_{D}^{+}(x_1)\neq \emptyset$ as $x_1$ is not sink. Take a vertex $z\in N_{D}^{+}(x_1)$. Since by Corollary \ref{inducedcomponent}, $D_1$ can not be an induced weighted oriented subgraph of $D$, we have $w(z)=1$. Now, we consider the subgraph $H$ of $G$ such that $I(H)=I(D)_{\ev{2}}$. Then $I(H)$ has a linear resolution, i.e. $H^c$ is chordal by Fr\"{o}berg's theorem. Clearly, $\{x,x_{1}\}\not\in E(H)$ as $w(x_1)>1$ and $(x,x_1)\in E(D)$. From our setup, it is easy to observe that $\{z,x_1\},\{x,x_2\},\ldots,\{x,x_r\},\{x,y_1\},\ldots,\{x,y_s\}\in E(H)$. Since $H^c$ has no induced cycle of length $4$ and $G$ is bipartite, we have $\mathcal{N}_{H}(z)=\{x_1,\ldots,x_r,y_1,\ldots,y_s\}=\mathcal{N}_{G}(z)$. Thus, generator of $I(D)$ containing $z$ is of degree $2$ and we can write $I(D)$ as follows:
     $$I(D)=z(x_1,\ldots,x_r,y_1,\ldots,y_r)+I(D\setminus z).$$
     Since $\mathcal{N}_{G}(x)=\mathcal{N}_{G}(z)$ is a minimal vertex cover of $G$, we have $(x_1,\ldots,x_r,y_1,\ldots,y_s)$ is a minimal prime of $I(D)$ and hence, $I(D\setminus z)\subseteq (x_1,\ldots,x_r,y_1,\ldots,y_s)$. By Proposition \ref{propcl}, $I(D\setminus z)$ is componentwise linear and so is vertex splittable by induction hypothesis. Therefore, from Definition \ref{spl1}, it follows that $I(D)$ is vertex splittable. 
 \end{proof}
 
\begin{definition}{\rm (Split Graph)
    A split graph is one whose vertex set can be partitioned into a disjoint union of an independent set and a clique (either of which may be empty).
    }
\end{definition}

\begin{theorem}\cite[Theorem 1]{sb98}\label{thmsplit}
Let $G$ be a graph. Then $G$ is a split graph if
and only if $G$ and $G^c$ are both chordal graphs.
\end{theorem}

\begin{theorem}\label{thmchordal}
     Let $D=(V(D), E(D), w)$ be a weighted oriented chordal graph. Then the following statements are equivalent.
     \begin{enumerate}
         \item $I(D)$ is componentwise linear;
         \item $I(D)$ is vertex splittable;
         \item $I(D)$ has linear quotient property.
     \end{enumerate}
\end{theorem}

\begin{proof}
    $(2)\implies (3) \implies (1)$ is clear. Thus, it is enough to show $(1)\implies (2)$.\par  
   \noindent $(1)\implies (2)$: We proceed by induction on the number of vertices $n$ of $D$. We may assume $n\geq 2$ and $D$ is not an empty graph as in the case of an empty graph, $I(D)=0$ trivially follows the equivalent statements. If $n=2$, then $I(D)$ has only one minimal monomial generator. Thus, by Definition \ref{spl1}, the result follows. Let $n>2$. Given that $G$ is chordal and by Theorem \ref{thmcochordal}, we have $G^c$ is chordal. Thus, $G$ is a split graph, i.e., we can partition $V(G)$ as $V(G)=A\cup B$, where $G[A]$ is a clique and $B$ is an independent set of $G$. Let $A=\{x_1,\ldots, x_r\}$ and $B=\{y_1,\ldots,y_s\}$. Note that $\mathcal{N}_{G}(x_i)$ is a minimal vertex cover of $G$ for all $1\leq i\leq r$. Also by Corollary \ref{inducedcomponent}, there exists a vertex $x\in A$ such that $w(x)=1$. Without loss of generality, let us assume $x=x_1$. By Corollary \ref{inducedcomponent}, $D_2$ and $D_4$ can not be a weighted oriented induced subgraph of $D$, we have $\vert \mathcal{N}_{D}^{+}(x_1)\cap V^{+}\vert\leq 1$. We may skip the case when $\vert \mathcal{N}_{D}^{+}(x_1)\cap V^{+}\vert= 0$, because, in this case, writing $I(D)=x_1(\mathcal{N}_{D}(x_1))+I(D\setminus x_1)$ we can easily observe that $I(D)$ is vertex splittable. Thus, we assume $\vert \mathcal{N}_{D}^{+}(x_1)\cap V^{+}\vert= 1$. If $\mathcal{N}_{D}^{+}(x_1)\cap V^{+}=\{y_i\}$ for some $1\leq i\leq s$, then we can write
    $$I(D)=x_{1}\left(y_{i}^{w(y_{i})},\mathcal{N}_{D}(x_1)\setminus \{y_i\}\right)+I(D\setminus x_1).$$
    Clearly, $A\setminus\{x_1\} \subseteq \mathcal{N}_{D}(x_1)\setminus \{y_i\}$ is a vertex cover of $G\setminus x_1$  and thus, we have $I(D\setminus x_1)\subseteq (y_{i}^{w(y_{i})},\mathcal{N}_{D}(x_1)\setminus \{y_i\})$. Therefore, by induction hypothesis and Lemma \ref{lemspl3}, it follows that $I(D)$ is vertex splittable. If $\mathcal{N}_{D}^{+}(x_1)\cap V^{+} \cap B=\emptyset$, then without loss of generality we may assume $\mathcal{N}_{D}^{+}(x_1)\cap V^{+}=\{x_2\}$. Now we will take care of two possible cases separately:\par
    
    \noindent \textbf{\underline{Case-1:}} Suppose $N_{D}^{+}(x_2)\subseteq {N}_{D}(x_1)$. Then we can write
    $$I(D)=x_{1}\left(x_{2}^{w(x_{2})},\mathcal{N}_{D}(x_1)\setminus \{x_2\}\right)+I(D\setminus x_1).$$
    From our assumption, it is clear that as $I(D\setminus x_1)\subseteq (x_{2}^{w(x_{2})},\mathcal{N}_{D}(x_1)\setminus \{x_2\})$. Therefore, by induction hypothesis and $(x_{2}^{w(x_{2})},\mathcal{N}_{D}(x_1)\setminus \{x_2\})$ being vertex splittable by Lemma \ref{lemspl3}, it follows that $I(D)$ is vertex splittable.\par

    \noindent \textbf{\underline{Case-2:}} Suppose $\mathcal{N}_{D}^{+}(x_2)\not\subseteq \mathcal{N}_{D}(x_1)$. Then there exists $y\in \mathcal{N}_{D}^{+}(x_2)\cap B$ such that $y\not\in \mathcal{N}_{D}(x_1)$. Since $w(x_2)>1$ and by Corollary \ref{inducedcomponent}, $D_1$ can not be an induced subgraph of $D$, we have $w(y)=1$. Now we will take three possible sub-cases: \par 
    \noindent \textbf{\underline{Case-2(a):}} Suppose $x_{2}x_{3},\ldots,x_{2}x_{r}\in I(D)$. If there exists $y'\in \mathcal{N}_{G}(x_1)\cap B$ such that $y'\not\in\mathcal{N}_{G}(x_2)$ then $x_1y'\in I(D)$ as $D_2$ can not be an induced subgraph of $D$. Thus, $\{x_1,y'\}$ and $\{x_2,y\}$ will form an induced matching in a subgraph $H$ of $G$ such that $I(H)=I(D)_{\ev 2}$. This implies that $G^c$ has a cycle of length four which is a contradiction by Theorem \ref{thmcochordal}. Therefore, $\mathcal{N}_{G}(x_1)\subseteq \mathcal{N}_{H}[x_2]$ and hence, we can write
    $$I(D)=x_2\left(\{zx_{2}^{w(x_2)-1}\mid z\in\mathcal{N}_{D}^{-}(x_2)\}\cup \mathcal{N}_{D}^{+}(x_2)\right)+I(D\setminus x_2),$$
    where $I(D\setminus x_2)\subseteq \left(\{zx_{2}^{w(x_2)-1}\mid z\in\mathcal{N}_{D}^{-}(x_2)\}\cup \mathcal{N}_{D}^{+}(x_2)\right)$. By induction hypothesis and Lemma \ref{lemspl3}, we get $I(D)$ is vertex splittable.\par 
    \noindent \textbf{\underline{Case-2(b):}} Suppose there exists $x'\in A$ such that $(x',x_2)\in E(D)$. Without loss of generality, let $x'=x_3$. Then by Corollary \ref{inducedcomponent}, we have $w(x_3)=1$. Similar like $x_1$, we have $\vert \mathcal{N}_{D}^{+}(x_3)\cap V^{+}\vert\leq 1$. Now, consider the graph $H=H(I(D)_{\ev 2})$. Since $\{x_2,y\},\{x_1,x_3\}\in E(H)$ and $\{y,x_1\},\{x_1,x_2\},\{x_2,x_3\}\not\in E(H)$ and $H^c$ is chordal, we have $y\in\mathcal{N}_{H}(x_3)$. Suppose, without loss of generality, there exists $y_1\in\mathcal{N}_{D}^{+}(x_2)\cap B$ such that $y_1\not\in \mathcal{N}_{D}(x_3)$. Then either $\{x_2,y_1\}$ and $\{x_1,x_3\}$ will form an induced matching in $H$ or $H[\{x_1,x_2,x_3,y,y_1\}]$ will be an induced $5$-cycle of $H$. In both cases, we get contradictions as $H^c$ is chordal. Thus, $\mathcal{N}_{D}^{+}(x_2)\subseteq \mathcal{N}_{D}(x_3)$ and we have 
    $$I(D)=x_3\left(x_{2}^{w(x_{2})}, \mathcal{N}_{D}(x_3)\setminus \{x_2\}\right)+I(D\setminus x_3),$$
    where $I(D\setminus x_3)\subseteq \left(x_{2}^{w(x_{2})}, \mathcal{N}_{D}(x_3)\setminus \{x_2\}\right)$. Hence, by induction and Lemma \ref{lemspl3}, $I(D)$ is vertex splittable.\par 
    \noindent \textbf{\underline{Case-2(c):}} Suppose there exists $x'\in A$ such that $(x_2,x')\in E(D)$ and $w(x')>1$. Without loss of generality, let us assume $x'=x_3$. Then due to Corollary \ref{inducedcomponent}, we have $\mathcal{N}_{D}^{-}(x_2)\subseteq \mathcal{N}_{D}^{+}(x_3)$, $\mathcal{N}_{D}(x_2)\cap V^{+}=\{x_3\}$. Let $y'\in \mathcal{N}_{D}^{+}(x_2)\cap B$ such that $y'\not\in \mathcal{N}_{D}^{+}(x_3)$. Then either $\{x_2,y'\}$ and $\{x_1,x_3\}$ will form an induced matching in $H$ or $H[\{x_1,x_2,x_3,y,y'\}]$ will be an induced $5$-cycle of $H$. In both cases, we get contradictions as $H^c$ is chordal. Thus, $\mathcal{N}_{D}(x_2)\setminus \{x_3\}\subseteq \mathcal{N}_{D}^{+}(x_3)$. If $x_3x_1,x_3x_4,\ldots,x_3x_r\in I(D)$, then we have
    $$I(D)=x_3\left(\{zx_{3}^{w(x_{3})-1}\mid z\in\mathcal{N}_{D}^{-}(x_3)\}\cup \mathcal{N}_{D}^{+}(x_3)\right)+I(D\setminus x_3)$$
    such that $I(D\setminus x_3)\subseteq (\{zx_{3}^{w(x_{3})-1}\mid z\in\mathcal{N}_{D}^{-}(x_3)\}\cup \mathcal{N}_{D}^{+}(x_3))$. By induction hypothesis and Lemma \ref{lemspl3}, we get $I(D)$ is vertex splittable. Now, if $(x_3,z)\in E(D)$ for some $z\in V^{+}$, then by looking at $G[\{x_2,x_3,z\}]$, we get a contradiction from Corollary \ref{inducedcomponent}. Therefore, the only case left is when there is some $x_i\in A\setminus \{x_2\}$ such that $(x_i,x_3)\in E(D)$. Note that $x_i=x_1$ is not possible by Corollary \ref{inducedcomponent}, and so, without loss of generality, we assume $x_i=x_4$. Similar like $x_1$, we have $\mathcal{N}_{D}^{+}(x_4)\cap V^{+}=\{x_3\}$. Suppose there exists $y'\in \mathcal{N}_{D}^{+}(x_3)\cap B$ such that $y'\not\in \mathcal{N}_{D}(x_4)$. Since $w(x_3)>1$ and $D_1$ can not be an induced subgraph of $D$, we have $w(y')=1$. Now let us focus on the graph $H=H(I(D)_{\ev 2})$. Since $H$ can not have any induced $5$-cycle and no two edges of $H$ can form an induced matching, we should have $\{y,x_3\},\{y,x_4\},\{y',x_1\},\{y',x_2\}\in E(H)$. Let $L=H[\{x_1,x_2,x_3,x_4,y,y'\}]$. Then, carefully looking at the graph $L$, we observe that $L^c$ is a $6$-cycle, which gives a contradiction as $H^c$ is chordal. Therefore, we have $\mathcal{N}_{D}^{+}(x_3)\subseteq \mathcal{N}_{D}(x_4)$ and we can write
    $$I(D)=x_4\left(x_{3}^{w(x_3)}, \mathcal{N}_{D}(x_4)\setminus \{x_3\}\right)+I(D\setminus x_4).$$
    From the above discussion it is clear that $I(D\setminus x_4)\subseteq (x_{3}^{w(x_3)}, \mathcal{N}_{D}(x_4)\setminus \{x_3\})$. Hence, by induction hypothesis and Lemma \ref{lemspl3}, $I(D)$ is vertex splittable.
\end{proof}

From Theorem \ref{thmsink}, \ref{thmcardinality}, \ref{thmbipartite} and \ref{thmchordal}, we observe that for some large classes of weighted oriented graphs, componentwise linearity, linear quotient property, and vertex splitting become equivalent.  Furthermore, through extensive computational analysis, we have not encountered any weighted oriented graph where these three properties are not equivalent. Consequently, we hold a strong belief that these three properties will be equivalent for the edge ideals of any weighted oriented graph. Hence, we propose the following conjecture.

\begin{conjecture}\label{conjcl}
Let $D$ be a weighted oriented graph. Then the following are equivalent:
\begin{enumerate}
    \item $I(D)$ is componentwise linear;
    \item $I(D)$ has linear quotient property;
    \item $I(D)$ is vertex splittable ideal.
\end{enumerate}
\end{conjecture}

\noindent
{\bf Acknowledgement:} 
Manohar Kumar is thankful to the Government of India for supporting him in this work through the Prime Minister Research Fellowship. Kamalesh Saha would like to thank the National Board for Higher Mathematics (India)
for the financial support through NBHM Postdoctoral Fellowship.

\end{document}